\documentclass[11pt]{article}
\usepackage{amssymb,amsfonts,amsthm, amsmath, color, float, mathtools, booktabs, hvfloat, url}
\usepackage{hyperref}
\hypersetup{pdfborder = {0 0 0}}
\usepackage{mathrsfs}
\usepackage{caption}
\def\obrace{\iftrue{\else}\fi}
\def\cbrace{\iffalse{\else}\fi}

\let\originalparagraph\paragraph
\renewcommand{\paragraph}[2][.]{\originalparagraph{#2#1}}

\newcommand{\zz}{{\mathbb Z}}

\newcommand{\rr}{{\mathbb R}}

\newcommand{\cale}{{\mathcal E}}

\DeclareMathOperator{\sign}{sign}

\newcommand{\call}{{\mathcal L}}

\newcommand{\beq}{\begin{eqnarray*}}
	\newcommand{\feq}{\end{eqnarray*}}
\newcommand{\beqn}{\begin{eqnarray}}
	\newcommand{\feqn}{\end{eqnarray}}
\newcommand{\bes}{\begin{split}}
	\newcommand{\fes}{\end{split}}

\newtheorem{theorem}{Theorem}
\newtheorem*{theorema*}{Theorem~A}
\newtheorem*{conj*}{Conjecture}
\makeatletter \@addtoreset{theorem}{section}\makeatother

\newcommand{\nn}{{\mathbb N}}

\makeatletter \@addtoreset{theorem}{section}\makeatother

\makeatletter \@addtoreset{theorem}{section}\makeatother

\newtheorem{lemma}[theorem]{Lemma}

\newtheorem*{theorem*}{Theorem}

\makeatletter
\newcommand{\leqnomode}{\tagsleft@true\let\veqno\@@leqno}
\newcommand{\reqnomode}{\tagsleft@false\let\veqno\@@eqno}
\makeatother
\usepackage{algorithm}
\usepackage{graphicx}
\usepackage[noend]{algpseudocode}
\makeatletter
\def\BState{\State\hskip-\ALG@thistlm}
\makeatother

\DeclareCaptionLabelFormat{noname}{#2}
\newlength\myindent
\title{On a characterization of exponential and \\ double exponential distributions}

\author{
	Reza~Rastegar\thanks{Occidental Petroleum Corporation, Houston, TX 77046, USA;  e-mail: reza\_rastegar2@oxy.com}
	\and
	Alexander~Roitershtein \thanks{Dept. of Mathematics, Texas A\&M University, College Station, TX 77843, USA;
		e-mail: roiterst@tamu.edu}
}

\begin{document}
\maketitle
\begin{abstract}
Recently, G.~Yanev \cite{yanev} obtained a characterization of the exponential family of distributions in terms of a functional equation for certain mixture densities. The purpose of this note is twofold: we extend Yanev's theorem by relaxing a restriction on the sign of mixture coefficients and, in addition, obtain a similar characterization for the Laplace family of distributions.
\end{abstract}
	
	
	\section{Introduction}
	\label{intro}
Our aim is to prove certain characterizations of the exponential and double exponential families of distributions. 
We will use the notation $X\sim \cale(\lambda)$ to indicate that $X$ is an exponential random variable with parameter $\lambda>0,$ that is $P(X>x)=e^{-\lambda x}$ for all $x>0.$ We will write $X\in \cale$ if $X\sim \cale(\lambda)$ for some $\lambda>0.$ Similarly, will write $X\in \call$ if $X$ has a Laplace (double exponential) distribution \cite{chara}, that is, for some $\lambda>0$ and $Y\sim\cale(\lambda),$ 
\beq
P(X>x)=\frac{1}{2}\Big(P(Y>x)+P(-Y>x)\Big)=\frac{\lambda}{2}\int_x^\infty e^{-\lambda |y|}\,dy, \qquad \forall\,x\in\rr.
\feq
For the exponential random variable we have:
	\begin{theorem}
		\label{yate}
		Let $X_1,\ldots, X_n,$ $n\geq 2,$ be independent copies of a random variable $X$ and
		$\mu_1,\ldots,\mu_n$ be distinct non-zero real numbers. Let $\varphi(t)=E\big(e^{itX}\big),$ $t\in\rr,$ be the characteristic function of $X,$ and suppose that $\varphi$ is infinitely differentiable at zero and, furthermore,
		\beqn
		\label{nusa}
		\prod_{k=1}^n \varphi(\mu_k t)=\sum_{k=1}^n \theta_k \varphi(\mu_k t),\qquad t\in\rr,
		\feqn
		where
		\beqn
		\label{theta}
		\theta_k=\prod_{j=1,j\neq k}^n\frac{\mu_k}{\mu_k-\mu_j},\qquad k=1,\ldots,n.
		\feqn
If, in addition, 
		\beqn
		\label{core}
		\sum_{(k_1,\ldots,k_n)\in W_{n,m}}\prod_{j=1}^n \mu_j^{k_j}\neq \sum_{k=1}^n \mu_k^m\quad \mbox{\rm for any integer $m\geq 2,$}
		\feqn
		where
		\beqn\label{wi}
		W_{n,m}:=\big\{(k_1,\ldots,k_n)\in\zz^n: k_j\geq 0~\mbox{\rm and}~\sum_{j=1}^n k_j=m\big\},
		\feqn		
then, either $P(X=0)=1$ or $E(X)\neq 0$ and $X\cdot \sign\big(E(X)\big)\sim \cale(\lambda)$ with $\lambda=1/E(X).$
	\end{theorem}
	The proof of the theorem is given in Section~\ref{proofs}. Theorem~\ref{yate} is an extension of a similar result of G.~Yanev \cite{yanev} obtained under the additional assumption that the coefficients $\mu_k$ are positive. In that case, the key technical condition \eqref{core} is trivial as the left-hand sides contains the $\mu_k^m$ terms and hence is always larger than the right-hand side.
	\par
	To ensure the existence of the derivatives of $\varphi$ at zero one can impose Cram\'{e}r's condition, namely assume that there is $t_0>0$ such that $E\big(e^{tX}\big) < \infty$ for all $t \in(-t_0, t_0).$ Note also that the equality in \eqref{core} for any fixed $m\in\nn$ describes a low-dimensional manifold in $\rr^n,$ and hence Theorem~\ref{yate} is true for almost every vector $(\mu_1,\ldots,\mu_n)$ chosen at random from a continuous distribution on $\rr^n.$
	\par
	The identity in \eqref{nusa} with $\theta_k$ introduced in \eqref{theta} holds for any $X\in\cale,$ and Theorem~\ref{yate} can be seen
	as a converse to this result. Equations \eqref{nusa} and \eqref{theta} give an expression of the characteristic function of the sum
\beqn
\label{es}
S=\mu_1X_1+\cdots \mu_n X_n
\feqn
as a linear combination of $\varphi(\mu_kt)$'s. If $X\sim \cale(\lambda),$ then $\varphi(t)=\frac{\lambda}{\lambda-it},$ and thus \eqref{nusa} 
is the partial fraction decomposition of the complex-valued rational function $\psi(t):=E\big(e^{itS}\big).$ In the particular case when $X\in\cale$ and $\mu_k=\frac{1}{L-k+1}$ for some integer $L>n,$ the random variable $S/\lambda$ is distributed as the $n$-th order statistic of a sample of $L$ independent copies of $X$ (this is the R\'{e}nyi representation of order statistics; see, for instance, \cite[p.~18]{david}). For further background and earlier versions (particular cases) of Yanev's characterization theorem see \cite{arnold, beluga, yanev}.
\par
It was pointed out in \cite{yanev} that an extension of their result to a more general class of coefficients $\mu_k$ would be of interest from the viewpoint of both theory and applications. When all the coefficients $\mu_k$ are positive and $X$ is an exponential random variable, the random variable $S=\sum_{k=1}^n \mu_k X_k$ has a hypoexponential distribution. When some of the coefficients are negative, $S$ is a difference of two hypoexponential random variables. Some applications of such differences are discussed, for instance, in \cite{li}. An insightful theoretical exploration of the densities of hypoexponential distributions can be fund in \cite{her}. 
\par 
We remark that the theorem is not true if the particular form of the coefficients $\theta_k$ in \eqref{theta} is not enforced.
For instance, for the Laplace distribution we have:
\begin{theorem}
\label{yan}
Let $X_1,\ldots, X_n,$ $n\geq 2,$ be independent copies of a random variable $X$
and $\mu_1,\ldots,\mu_n$ be distinct positive numbers. Let $\varphi(t)=E\big(e^{itX}\big),$ $t\in\rr,$
be the characteristic function of $X,$ and suppose that $\varphi$ is infinitely differentiable at zero and, furthermore, \eqref{nusa} holds with
\beqn
\label{theta}
\theta_k=\prod_{j=1,j\neq k}^n\frac{\mu_k^2}{\mu_k^2-\mu_j^2},\qquad k=1,\ldots,n.
\feqn
Then, either $P(X=0)=1$ or $X$ has a Laplace distribution.
\end{theorem}  
The result is closely related to the one stated in Theorem~\ref{yate} because $X\in \call$ implies that for a suitable $Y\in\cale,$
\beq
E\big(e^{it X}\big)=\frac{1}{2}\Big(E\big(e^{it Y}\big)+E\big(e^{-it Y}\big)\Big).
\feq 
The proof of the theorem is similar to that of Theorem~\ref{yate}, and therefore is omitted. The key technical ingredient of the proof, namely an analogue of Lemma~\ref{lama} for Laplace distributions, follows immediately from Lemma~2-(iii) in \cite{yanev}, and the rest of the proof of Theorem~\ref{yate} can be carried over verbatim to the double exponential setup of Theorem~\ref{yan}.
\par
We conclude the introduction with a brief discussion of condition \eqref{wd}. The equality with $n=2$ and some $m\geq 2$ reads 
$\sum_{j=0}^m \mu_1^j\mu_2^{m-j}=\mu_1^m+\mu_2^m, $ which is equivalent to $\frac{\mu_1^{m+1}-\mu_2^{m+1}}{\mu_1-\mu_2}=\mu_1^m+\mu_2^m.$ 
The latter implies that $\mu_2^{m-1}=\mu_1^{m-1},$ and hence $m$ is odd and $\mu_2=-\mu_1.$ In that case, \eqref{nusa} becomes
\beqn
\label{phisa}
\varphi(t)\varphi(-t)=\frac{1}{2}\big(\varphi(t)+\varphi(-t)\big), \qquad t\in \rr.
\feqn 
The equation is satisfied when $X$ is a Bernoulli random variable with $P(X=0)=P(X=a)=\frac{1}{2}$ for some constant $a>0,$ 
in which case $\varphi(t)=\frac{1}{2}\big(1+e^{iat}\big).$ More generally, \eqref{phisa} holds if and only if $\varphi(t)=\frac{1}{2}\big(1+e^{i\rho(t)}\big),$ where $\rho:\rr\to\rr$ is an odd function. 
Unfortunately, we are not aware of any example where $\varphi$ in this form would be a characteristic function of a random variable beyond the linear case $\rho(t)=at$ and linear fractional $\rho(t)=\frac{\lambda-ti}{\lambda+ti},$ $\rho(t)=\frac{\lambda+ti}{\lambda-ti}$ which correspond to, respectively, $X\in \cale(\lambda)$ and $-X\in \cale(\lambda).$ 
\par 
Our proof technique differs significantly from the one used in \cite{yanev}. However, interestingly enough, both rely on the validity of \eqref{core}. Nevertheless, we believe that the following might be true:
\begin{conj*} 
For $n\geq 3,$ \eqref{core} is an artifact of the proof and is not necessary.
\end{conj*}
\section{Proof of Theorem~\ref{yate}}
\label{proofs}
The following is a suitable version of Lemma~2-(iii) in \cite{yanev}.
\begin{lemma}
\label{lama}
Assume \eqref{core}. Then, for any integer $m\geq 2,$
\beq
\sum_{k=1}^n \theta_k \mu_k^m \neq \sum_{k=1}^n \mu_k^m.
\feq
\end{lemma}
\begin{proof}[Proof of Lemma~\ref{lama}] 
Assume $X\in \cale$ and recall $S$ from \eqref{es}. It follows from \eqref{nusa} that 
\beq
E(S^m)=\sum_{k=1}^n \theta_k \mu_k^m E\big(X_1^m\big)=\frac{m!}{\lambda^m}\sum_{k=1}^n \theta_k \mu_k^m,
\feq 
and hence, with $W_{n,m}$ introduced in \eqref{wi}, we have:
\beqn
\nonumber 
\sum_{k=1}^n \theta_k \mu_k^m&=&\frac{\lambda^m}{m!}E(S^m)=\sum_{(k_1,\ldots,k_n)\in W_{n,m}}\frac{\lambda^m}{k_1!\cdots k_n!}
\prod_{j=1}^n \mu_j^{k_j} E\big(X_j^{k_j}\big)
\\
\label{wd}
&=&\sum_{(k_1,\ldots,k_n)\in W_{n,m}}\prod_{j=1}^n \mu_j^{k_j}, 
\feqn 
which yields the result in view of \eqref{core}.
\end{proof}
Differentiating both sides of \eqref{nusa} $m$ times we obtain the identity
\beqn
\label{ida}
\frac{d^m}{dt^m}\prod_{k=1}^n \varphi(\mu_k t)\,\bigg|_{t=0}=\sum_{k=1}^n \theta_k \mu_k^m \varphi^{(m)}(0),\qquad m\geq 2.
\feqn 
In view of Lemma~\ref{lama} and the fact that $\varphi(0)=1,$ these identities can be used 
to determine all the derivatives of $\varphi$ at zero in terms of $\varphi'(0),$ first $\varphi''(0)$ in terms of the parameter $\varphi'(0),$ then $\varphi'''(0)$ in terms of $\varphi'(0)$ and $\varphi''(0),$ and hence in terms of $\varphi'(0)$ only, and so on. For instance, \eqref{ida} with $m=2$ yields
\beq 
\varphi''(0)\sum_{k=1}^n \mu_k^2(\theta_k-1)=\varphi'(0)\sum_{k=1}^{n-1}\sum_{j=i+1}^n \mu_i\mu_j.
\feq 
Let now $Z\in \cale$ and $\psi(t)=E\big(e^{itZ}\big).$ The derivatives $\psi''(0), \psi'''(0),\ldots$ as functions of the parameter $\psi'(0)$ can be in principle derived using the same inductive algorithm. Therefore, $\varphi'(0)=0$ implies $P(X=0)=1$ while $\varphi'(0)=\psi'(0)=\lambda^{-1}$ for some $\lambda>0$ implies that $\varphi^{(m)}(0)=\psi^{(m)}(0)$ for all $m\in\nn,$ and hence (since $\varphi$ is analytic under the conditions of the theorem) $\varphi(t)=\psi(t)=\frac{\lambda}{\lambda-it}$ as desired.
Finally, the case $\varphi(0)=-\lambda^{-1}<0$  can be reduced to the previous one by switching from $X$ to $-X$ in the above argument. \hfill\hfill \qed
	
	\end{document}